\documentclass[10pt,fleqn]{article}

\usepackage{amsxtra,amssymb,amsthm,latexsym}
\usepackage{amsmath}
\usepackage{graphicx}
\newtheorem{thm}{Theorem}[section]

\newtheorem{lem}[thm]{Lemma}
 \newcommand{\thmref}[1]{Theorem~\ref{#1}}
 \newcommand{\lemref}[1]{Lemma~\ref{#1}}

\newcommand{\pt}{\partial}

\newcommand{\R}{{\mathbb R}}
\newcommand{\C}{{\mathbb C}}
\newcommand{\N}{{\mathbb N}}
\newcommand{\Lp}{\triangle}

\newcommand{\dl}{{\delta}}
\newcommand{\bee}{\begin{equation*}}
\newcommand{\eee}{\end{equation*}}
\newcommand{\be}{\begin{equation}}
\newcommand{\ee}{\end{equation}}
\newcommand{\pn}{\par\noindent}
\title{Creating materials with a desired refraction
coefficient: numerical experiments }
\author{Sapto W. Indratno and Alexander G. Ramm$^*$\\
\small Department of Mathematics\\[-0.8ex]
\small Kansas State University, Manhattan, KS 66506-2602, USA\\
\small \texttt{sapto@math.ksu.edu}\\
\small $*$Corresponding author:  \texttt{ramm@math.ksu.edu}}
\date{}
\begin{document}

\maketitle

\begin{abstract}

A recipe for creating materials with a desired refraction
coefficient is implemented numerically. The following assumptions
are used: \bee \zeta_m=h(x_m)/a^\kappa,\quad
d=O(a^{(2-\kappa)/3}),\quad M=O(1/a^{2-\kappa}),\quad
\kappa\in(0,1), \eee where $\zeta_m$ and $x_m$ are the boundary
impedance and center of the $m$-th ball, respectively, $h(x)\in
C(D)$, Im$h(x)\leq 0$, $M$ is the number of small balls embedded in
the cube $D$, $a$ is the radius of the small balls and $d$ is the
distance between the neighboring balls.

An error estimate is given for the approximate solution of the many-body
scattering problem in the case of small scatterers. This result
is used for the estimate of the minimal number of small particles
to be embedded in a given domain $D$ in order to get a material
whose refraction coefficient approximates the desired one with the
relative error not exceeding a desired small quantity.

\end{abstract}
\pn{\\ {\em MSC: 65R20, 65Z05, 74Q10}\\
{\em Key words:} many-body wave scattering problem, metamaterials,
refraction coefficient }
\section{Introduction}
A theory of wave scattering by many small bodies embedded in a
bounded domain $D$ filled with a material with known refraction
coefficient was developed in \cite{RAMM1}-\cite{RAMM4}. It was
assumed in \cite{RAMM1} that \bee d=O(a^{1/3}),\quad
M=O(a^{-1}),\quad \frac{\partial u_M}{\partial \nu}=\zeta_mu_M\quad
on \quad S_m,\quad 1\leq m\leq M, \eee where $a$ is the
characteristic size of the small particles, $d$ is the distance
between two neighboring particles, $M$ is the total number of the
embedded particles, $S_m$ is the boundary of $m$-th particle $D_m$,
$\nu$ is the unit normal to $S_m$ directed out of $D_m$, and
$\zeta_m=h_m/a$, where $h_m$, Im$h_m\leq 0$, $1\leq m\leq M$, are
constants independent of $a$.

Let us assume that $D$ is filled with a material with known
refraction coefficient $n_0^2(x)$, Im$n_0^2(x)\geq 0$, $n_0^2(x)=1$
in $D':=\R^2\setminus D$, $n_0^2(x)$ is Riemann integrable. The
governing equation is \be\label{e1} L_0u_0:=[\Lp
+k^2n_0^2]u_0=0,\quad \text{in } \R^3, \ee \be\label{e2}
u_0=\exp(ikx\cdot \alpha)+v_0 \ee where $k$ is the wave number,
$\alpha\in S^2$ is the direction of the incident plane wave, $S^2$
is the unit sphere in $\R^3$,
 and $v_0$ is the scattered
field satisfying the radiation condition \be\label{e3} \lim_{r\to
\infty}r\left(\frac{\pt v_0}{\pt r}-ikv_0\right)= 0\qquad r:=|x|\to
\infty, \ee and the limit is attained uniformly with respect to the
directions $x^0:=x/r$.

Let $n^2(x)$ be a desired refraction coefficient in $D$. We assume
that $n^2(x)$ is  Riemann integrable, Im $n^2(x)\geq 0$, $n^2(x)=1$
in $D'$. Our objective is to create materials with the refraction
coefficient $n^2(x)$ in $D$ by embedding into $D$ many small
non-intersecting balls $B_m$, $1\leq m\leq M$, of radius $a$,
centered at some points $x_m\in D$. If one embeds $M$ small
particles $B_m$ in the bounded domain $D$, then the scattering
problem consists of finding the solution to the following problem:
\be\label{e4} L_0u_M:=[\Lp +k^2n_0^2]u_M(x)=0\qquad x\in
\R^3\setminus \cup_{m=1}^M B_m, \ee \be\label{e5} \frac{\partial
u_M}{\partial\nu}=\zeta_mu_M\quad on\ S_m:=\partial B_m,\quad 1\leq
m\leq M, \ee \be\label{e6} u_M=u_0+v_M, \ee where $u_0$ solves
problem \eqref{e1}-\eqref{e3}, and $v_M$ satisfies the radiation
condition.

The following theorem is proved in \cite{RAMM4} under the
assumptions \be\label{e7} \zeta_m=h(x_m)/a^\kappa,\quad
d=O(a^{(2-\kappa)/3}),\quad M=O\left(1/a^{2-\kappa}\right),\quad
\kappa\in(0,1), \ee where $h(x)$ is a continuous function in $D$,
$Im h\leq 0$, and $\kappa\in(0,1)$ is a parameter one can choose as
one wishes. Below it is always assumed that conditions \eqref{e7}
hold.
\begin{thm}[\cite{RAMM4}]\label{thm1}
Assume that conditions \eqref{e7} are satisfied, and $D_m$ is a
ball of radius $a$ centered at a point $x_m$. Let $h(x)$ in
\eqref{e7} be an arbitrary continuous function in $D$, Im$h(x)\leq
0$, $\Delta_p\subset D$ be any subdomain of $D$, and
$\mathcal{N}(\Delta_p)$ be the number of particles in $\Delta_p$,
\be\label{e8}\mathcal{N}(\Delta_p)=\frac{1}{a^{2-\kappa}}\int_{\Delta_p}N(x)dx[1+o(1)],\quad
a\to 0, \ee
 where $N(x)\geq 0$ is a given continuous function in $D$. Then
\be\label{e9} \lim_{a\to 0}\|u_e(x)-u(x)\|_{C(D)}=0, \ee where
\be\label{e10}
u_e(x):=u_0(x)-4\pi\sum_{j=1}^MG(x,x_j)h(x_j)u_e(x_j)a^{2-\kappa}[1+o(1)],\quad
a\to 0, \ee where $\min_j|x-x_j|\geq a$, and $G(x,y)$ is the Green
function of the operator $L_0$ in $\R^3$, $G(x,y)$ satisfies the
radiation condition. The numbers $u_e(x_j)$, $1\leq j\leq M$, are
found from the linear algebraic system: \be\label{e11}
u_e(x_m)=u_0(x_m)-4\pi\sum_{j=1, j\neq
m}^MG(x_m,x_j)h(x_j)u_e(x_j)a^{2-\kappa},\ \qquad m=1,2,\hdots, M,
\ee which is uniquely solvable for all sufficiently large $M$. The
function $u(x)=\lim_{a\to 0} u_e(x)$ solves the following limiting
equation: \be\label{e12} u(x)=u_0(x)-\int_D G(x,y)p(y)u(y)dy,\ee
where $u_0$ satisfies equations \eqref{e1}-\eqref{e3},
\be\label{e13} p(x):=4\pi N(x)h(x), \ee \be\label{e14}
n^2(x):=1-k^{-2}q(x), \ee \be\label{e15} q(x):=q_0(x)+p(x),\quad
q_0(x):=k^2-k^2n_0^2(x), \ee and $n_0^2(x)$ is the coefficient in
\eqref{e1}.
\end{thm}

In \cite{RAMM4} a recipe for creating material with a desired
refraction coefficient is formulated.

The goal of this paper is to implement numerically the recipe for
creating materials with a desired refraction coefficient in a given
domain $D$ by embedding in $D$ many small particles with prescribed
physical properties. These particles are balls of radius $a$,
centered at the points $x_m\in D$, and their physical properties are
described by the boundary impedances $\zeta_m=h(x_m)/a^\kappa$. A
formula for embedding the small balls in $D$ is given in Section 2.

We give an estimate for the  error in the refraction coefficient of
the medium obtained by embedding finitely many $(M<\infty)$ small
particles, compared with the refraction coefficient of the limiting
medium $(M\to \infty)$. This is important because in practice one
cannot go to the limit $M\to \infty$, i.e., $a\to 0$,  and one has
to know the maximal $a$ (i.e., minimal $M$)  such that the
corresponding to this $a$ refraction coefficient differs from the
desired refraction coefficient by not more than a given small
quantity.
 In Section 3 we give an algorithm for finding the
minimal number $M$ of the embedded small balls which generate a
material whose refraction coefficient differs from a desired one by
not more than a desired small quantity. In Section 4 some numerical
experiments are described.

\section{Embedding small balls into a cube}

In this section we give a formula for distributing small balls in a
cube in such a way that the second and third restrictions \eqref{e7}
are satisfied.

Without loss of generality let us assume that the domain $D$ is the unit
cube:
\be\label{e16} D:=[0,1]\times[0,1]\times[0,1].\ee Let \be\label{e17}
D=\cup_{q=1}^{n^3}\overline{\Delta_q}, \ n\in \N,\quad
\Delta_i\cap\Delta_j=\emptyset\quad for\quad i\neq j, \ee where $\N$
is the set of positive integers, $\overline{X}$ is the closure of
the set $X$, and $\Delta_q$, $q=1,2,\hdots,n^3$, are cubes of side
length $1/n$.

{\it Definition:} {\it We say that $D$ has property
$Q_n$ if each small cube $\Delta_q$ contains a ball of radius $a_n$,
$0<a_n<1/n$, centered at the centroid of the cube $\Delta_q$, and
the following condition holds \be\label{e18} d_n:=\min_{q\neq j}
\text{dist}(B_{a_n}(x_q),B_{a_n}(x_j))=\gamma a_n^{(2-\kappa)/3},
\ee where $x_q$ is the centroid of the cube $\Delta_q$,
$q=1,2,\hdots,n^3$, \be\label{e19} B_{a}(x):=\{y\in \R^3\ |\
|y-x|<a\}, \ee and $\gamma>0$ is a constant which is not too small
(see formula \eqref{e25}).}

From \eqref{e17} and \eqref{e18} one gets
\be\begin{split}\label{e20} d_n&=l_n-2a_n=\gamma
a_n^{(2-\kappa)/3},\quad l_n:=1/n.\end{split}\ee Since $l_n=1/n$,
the quantity $a_n$ solves the equation \be\label{e21} \gamma
a^{(2-\kappa)/3}+2a-1/n=0. \ee The function $f(a):=\gamma
a^{(2-\kappa)/3}+2a$ is strictly growing on $[0,\infty)$. Thus,  the
solution to equation \eqref{e21} exists, is unique, and can be
calculated numerically, for example, by the bisection method.

However, it is easy to derive an analytic asymptotic formula for
$a_n$ as $n\to \infty$. This formula is simple and can be used for
all $n$ we are interested in, since these $n$ are sufficiently large.

Let us derive this asymptotic formula. Since
$1/3<(2-\kappa)/3<2/3$, one has $a\ll a^{(2-\kappa)/3}$ if $a\ll 1$.
Therefore, \be\label{e22}
a_n=[1/(n\gamma)]^{3/(2-\kappa)}[1+o(1)],\quad as\ n\to \infty, \ee
is the desired asymptotic formula for the solution to \eqref{e21}. Note
that
\be\label{e23} \lim_{n\to\infty}a_n=0\quad and \quad
\lim_{n\to\infty}na_n=0,\ee
as follows from \eqref{e22} because $3/(2-\kappa)>1$.

 We note that $a_n/d_n\ll 1,$ if $ n>>1$, because \eqref{e20}
yields
\be\label{e24} a_n/d_n=a_n/(\gamma a_n^{(2-\kappa)/3})=
a_n^{(1+\kappa)/3}/\gamma\ll 1.
\ee
Let us
choose $n$ sufficiently large so that \be\label{e25} \gamma\gg
(l_n/2)^{(1+\kappa)/3},\qquad \kappa\in(0,1),\qquad l_n=1/n,\ee and
make the following
assumption:\\
\textit{Assumption A): $D$ has property $Q_{mP}$. Here
$D=\cup_{q=1}^{P^3}\overline{\Omega_q}$,
$\Omega_j\cap\Omega_i=\emptyset$ for $j\neq i$, where each cube
$\Omega_q$ has side length $1/P$, and $m^3$ small balls are embedded
in $\Omega_q$ so that the following two conditions hold:
\begin{enumerate}
\item[1.] Each cube $\Omega_q$ is a union of small
sub-cubes $\Delta_{j,q}$: \be\label{e26}
\Omega_q=\cup_{j=1}^{m^3}\Delta_{j,q},\quad
\Delta_{i,q}\cap\Delta_{j,q}=\emptyset\quad for\quad i\neq j,
\ee where $\Delta_{j,q}$, $j=1,2,\hdots,m^3$,
$q=1,2,\hdots,P^3$, are cubes of side length $1/(mP)$,
\item[2.] In each sub-cube $\Delta_{j,q}$ there is a ball of radius
$a_{mP}$, $0<a_{mP}<1/(mP)$, centered at the centroid of the
sub-cube $\Delta_{j,q}$, and  the radius $a_{mP}$ of the
embedded balls satisfies the relation \be\label{e27}
1/(mP)-2a_{mP}=\gamma a_{mP}^{(2-\kappa)/3},\qquad \gamma\gg
[1/(2mP)]^{(\kappa+1)/3},
 \ee
\end{enumerate}
where $\gamma>0$ is a fixed constant.}
\begin{lem}\label{lem21}If Assumption A) holds, then \be\label{e28}
\lim_{m\to \infty}M a_{mP}^{2-\kappa}=1/\gamma^3, \ee
where $M=(mP)^3$ is the total number of small balls embedded in the
unit cube $D$, $\gamma>0$ is fixed, and
\be\label{e29}
a_{mP}=[1/(\gamma mP)]^{3/(2-\kappa)}[1+o(1)]\quad as\ m\to \infty.
\ee
\end{lem}
\begin{proof} Relation \eqref{e29} is an immediate consequence of
\eqref{e27}. Using this relation, one obtains
\be\begin{split}\label{e30} \lim_{m\to
\infty}Ma_{mP}^{2-\kappa}&=\lim_{m\to
\infty}(mP)^{3}a_{mP}^{2-\kappa}\\
&=\lim_{m\to \infty}\left(mP\right)^3[1/(\gamma
mP)]^{3}[1+o(1)]\\
&=\lim_{m\to \infty}(1/\gamma^3)[1+o(1)]=1/\gamma^3.\end{split}\ee
\lemref{lem21} is proved.
\end{proof}

\section{A recipe for creating materials with a desired refraction coefficient}
In this section the recipe given in \cite{RAMM4} is used for creating
materials with a desired refraction coefficient
by embedding into $D$ small balls so that Assumption A) holds.

{\it Step 1.} Given the refraction coefficient $n_0^2(x)$ of the
original material in $D$ and the desired refraction coefficient
$n^2(x)$ in $D$, one calculates \be\label{e31}
p(x)=k^2[n_0^2(x)-n^2(x)]=p_1(x)+ip_2(x), \ee where \bee p_1:=\text{
Re }p(x) \quad \text{\,\, and\,\,\, }  p_2(x):=\text{ Im }p(x).\eee

Choose \be\label{e32} N(x)=1/\gamma^3,\ee where $\gamma$ is
the constant $\gamma$ in Assumption A).

{\it Step 2.} Choose \be\label{e33} h(x)=h_1(x)+ih_2(x),\ee where
the functions $h_1(x)$ and $h_2(x)$ are defined by the formulas:
\be\label{e34} h_i(x)=\gamma^3 p_i(x)/(4\pi),\qquad i=1,2, \ee
and the functions $p_i(x)$ are defined in {\it Step 1.}

{\it Step 3.} Partition  $D$ into $P$ small cubes $\Omega_p$ with
side length $1/P$, and embed $m^3$ small balls in each cube
$\Omega_p$ so that Assumption A) holds.

Then \be\label{e35}
\mathcal{N}(\Omega_p)=\frac{1}{a_{mP}^{2-\kappa}}\int_{\Omega_p}N(x)dx=
|\Omega_p|/(\gamma^3 a_{mP}^{2-\kappa})=1/[\gamma
Pa_{mP}^{(2-\kappa)/3}]^3, \ee
where $\mathcal{N}(\Delta_p)$ is the
number of the balls embedded in the cube $\Omega_p$, $\kappa\in(0,1)$,
$a_{mP}$ is the radius of the embedded balls, and $|\Omega_p|$ is the
volume of the cube $\Omega_p$.

Since \bee a_{mP}=[1/(m\gamma P)]^{\frac{3}{2-\kappa}}[1+o(1)]
\qquad as \qquad m\to \infty,\eee it follows that \be\label{e36}
\lim_{m\to \infty}\frac{\mathcal{N}(\Omega_p)}{m^3}=1. \ee

By Assumption A) the balls are situated at the distances $\gamma
a_{mP}^{\frac{2-\kappa}{3}}, $ $\gamma>\left[1/(mP)\right]^
{\frac{1+\kappa}{3}}.$ Therefore, all the assumptions, made in
\thmref{thm1}, hold. Thus, \be\label{e37} \max_{x\in
D}|u_e(x)-u(x)|\to 0\quad \text{as } M\to \infty, \ee where $u_e(x)$
is defined in \eqref{e10} and $u(x)$ solves \eqref{e12}. Let us
assume for simplicity that $n_0^2(x)=1$, so that
$$G(x,y)=g(x,y):=\exp(ik|x-y|)/(4\pi |x-y|).$$
Then \be\label{e38}
u_e(x)=u_0(x)-4\pi\sum_{j=1}^{M}g(x,x_j)h(x_j)u_e(x_j)a_{mP}^{2-\kappa},
\quad |x-x_j|>a_{mP}, \ M:=(mP)^3,\ee and  the limiting function
\bee u(x)=\lim_{M\to \infty}u_e(x)\eee solves the integral equation
\be\label{e39} u(x)+Tu(x)=u_0(x), \ee where \be\label{e40}
Tu(x):=\int_Dg(x,y)p(y)u(y)dy, \ee \be\label{e41}
g(x,y):=\exp(ik|x-y|)/(4\pi |x-y|), \ee $M:=(mP)^3$, $h(x)=h_1(x)+i
h_2(x)$, $h_i(x),i=1,2,$ are defined in \eqref{e34},
\be\begin{split}\label{e42}
p(x)&=k^2[n_0^2(x)-n^2(x)]=4\pi
[h_1(x)+ih_2(x)]N(x)=4\pi [h_1(x)+ih_2(x)]/\gamma^3,\end{split}\ee
and the function $u_0(x)$ in  \eqref{e38} solves the scattering
problem \eqref{e1}-\eqref{e3}.

It follows from \eqref{e37} that
\be\label{e43}\max_{1\leq l\leq
M}|u(x_l)-u_e(x_l)|\to 0\text{ as }M\to \infty,\ee where $u_e(x)$
and $u(x)$ are defined in \eqref{e38} and \eqref{e39}, respectively.
Here and throughout this paper $D:=\cup_{j=1}^{M} D_j$, $M:=(mP)^3$,
$D_j$  ($j=1,2,\hdots,M$) are cubes with the side length $1/(mP)$,
$D_j\cap D_l=\emptyset$ for $j\neq l$, and $x_j$ denotes the center
of the cube $D_j$. However, since \eqref{e37} was not proved here,
let us prove relation \eqref{e43}. We denote
$$\|u\|_\infty:=\sup_{x\in D}|u(x)|\quad \text{and } \|v\|_{\C^M}:=\max_{1\leq
j\leq M}|v_j|,\ v:=\left(
                     \begin{array}{c}
                       v_1 \\
                       v_2 \\
                       \vdots \\
                       v_M \\
                     \end{array}
                   \right)\in \C^M,$$ where $\C$ is the
set of complex numbers.

Consider the following piecewise-constant function as an approximate
solution to equation \eqref{e39}: \be\label{e44}
u_{(M)}(x):=\sum_{j=1}^M \chi_j(x)u_{j,M}, \ee where $u_{j,M}$
$(j=1,2,\hdots,M)$ are constants and \be\label{e45}
\chi_j(x):=\left\{\begin{array}{ll}
                                                             1, & \hbox{$x\in D_j$,} \\
                                                             0, & \hbox{otherwise.}
                                                           \end{array}
                                                         \right.\ee
Substituting $u_{(M)}(x)$ for $u(x)$ in \eqref{e39} and evaluating
at points $x_l$, one gets the following linear algebraic system
(LAS) which is used to find the unknown $u_{j,M}$: \be\label{e46}
\tilde{u}_{(M)}+T_{d,M}\tilde{u}_{(M)}=u_{0,M}, \ee
where $T_{d,M}$ is a discrete version of $T_M$, defined below,
\be\label{e47} \tilde{u}_{(M)}:=\left(
                                  \begin{array}{c}
                                    u_{1,M} \\
                                    u_{2,M} \\
                                    \vdots \\
                                    u_{M,M} \\
                                  \end{array}
                                \right)\in \C^M,\quad u_{0,M}:=\left(
                                  \begin{array}{c}
                                    u_0(x_1) \\
                                    u_0(x_2) \\
                                    \vdots \\
                                    u_0(x_M) \\
                                  \end{array}
                                \right)\in \C^M,\ee
$u_0(x)$ solves problem (1)-(3), and \be\label{e48}
(T_{d,M}v)_l:=\sum_{j=1}^{M}\int_{D_j} g(x_l,y)p(y)dy v_j,\quad
l=1,2,\hdots,M,\ v:=\left(
                                  \begin{array}{c}
                                    v_1 \\
                                    v_2 \\
                                    \vdots \\
                                    v_M \\
                                  \end{array}
                                \right)\in \C^M. \ee Multiplying the
                                $l$-th equation of \eqref{e46} by $\chi_l(x)$, $l=1,2,\hdots,M,$ and summing up over $l$ from $1$ to $M$,
one gets
 \be\label{e49}
u_{(M)}(x)=u_{0,(M)}(x)-T_Mu_{(M)}(x), \ee
where $u_{(M)}(x)$ is
defined in \eqref{e44}, \be\label{e50}
u_{0,(M)}(x):=\sum_{j=1}^M\chi_j(x)u_0(x_j), \ee
and
\be\label{e51}
T_M u(x):=\sum_{j=1}^{M}\chi_j(x)\int_D g(x_j,y)p(y)u(y)dy,\quad
M=(mP)^3. \ee  It was proved in \cite{RAMM563} that equation
\eqref{e49} is equivalent to \eqref{e46} in the sense that
$\{u_{j,M}\}_{j=1}^M$ solves \eqref{e46} if and only if function
\eqref{e44} solves \eqref{e49}.

\begin{lem}\label{lem31}
For all sufficiently large $M$ equation \eqref{e46} has
a unique solution, and there exists a constant $c_1>0$ such that
\be\label{e52} \|(I_{d,M}+T_{d,M})^{-1}\|\leq c_1,\quad \forall
M>M_0, \ee
where $M_0>0$ is a sufficiently large number.
\end{lem}
\begin{proof}
Consider the operators $T$ and $T_M$ as operators in the space
$L^\infty(D)$ with the $\sup$-norm. Let
$\|(T-T_M)u\|_\infty:=\sup_{x\in D}|(T-T_M)u(x)|$. Then
\be\begin{split}\label{e53} \|(T-T_M)u\|_\infty&\leq
\max_i\sup_{x\in
D_i}\sum_{j=1}^M\int_{D_j}\left|(g(x,y)-g(x_i,y))p(y)u(y)\right|dy\\
&\leq \|u\|_\infty\|p\|_\infty\max_i\sup_{|x-x_j|\leq
\frac{1}{mP}}\sum_{j=1}^M\int_{D_j}|g(x,y)-g(x_i,y)|dy\\
&\leq O(1/(mP)).
\end{split}\ee
This implies \be\label{e54} \|T-T_M\|=O(1/(mP))=O(1/M^{1/3})\to 0
\text{ as } M\to \infty. \ee The operator $I+T$ is known to be
boundedly invertible, so $\|(I+T)^{-1}\|<c$, where $c>0$ is a
constant. Therefore,
 \be\label{e55}
I+T_M=(I+T)[I+(I+T)^{-1}(T_M-T)].
 \ee By \eqref{e54} there exists $M_0$ such that
 \be\label{e56}
\|(I+T)^{-1}(T_M-T)\|\leq c\|T_M-T\|<\delta<1,\quad \forall M>M_0,
 \ee where $\dl>0$ is a constant.
From \eqref{e56} we obtain, $\forall M>M_0$,\be\label{e57}
\|[I+(I+T)^{-1}(T_M-T)]^{-1}\|\leq
\frac{1}{1-\|(I+T)^{-1}(T_M-T)\|}\leq 1/(1-\dl). \ee
 Therefore, it follows from
 \eqref{e55} that $I+T_M$ is boundedly invertible and
 \be\label{e58}
(I+T_M)^{-1}=[I+(I+T)^{-1}(T_M-T)]^{-1}(I+T)^{-1},
 \ee so there exists a constant $c_0>0$ such that \be\label{e59}
\|(I+T_M)^{-1}\|\leq c_0 , \quad \forall M>M_0.\ee Since \eqref{e49}
is equivalent to \eqref{e46}, it follows that the homogeneous
equation $v+T_{d,M}v=0$ has only trivial solution for $M>M_0$, i.e.,
$\mathcal{N}(I_{d,M}+T_{d,M})=\{0\}$ for $M>M_0$, where
$\mathcal{N}(A)$ is the nullspace of the operator $A$, $I_{d,M}$ is
the identity operator in $\C^M$ and $T_{d,M}$ is defined in
\eqref{e48}. Therefore, by the Fredholm alternative equation
\eqref{e46} is solvable for $M>M_0$. This together with \eqref{e59}
yield the existence of a constant $c_1>0$ such that
$\|(I_{d,M}+T_{d,M})^{-1}\|\leq c_1$ for $M>M_0$.\\
\lemref{lem31} is proved.
\end{proof}

Define $T_d:C^2(D)\to \C^M$ as follows \be\label{e60}
(T_dw)_l:=(Tw)(x_l)=\sum_{j=1}^M\int_{D_j}g(x_l,y)p(y)w(y)dy,\quad
l=1,2,\hdots,M, \ee and \be\label{e61}u_M:=\left(
                                  \begin{array}{c}
                                    u(x_1) \\
                                    u(x_2) \\
                                    \vdots \\
                                    u(x_M) \\
                                  \end{array}
                                \right)\in \C^M,\ee where
$T$ is defined in \eqref{e40} and $u(x)$ solves \eqref{e39}. Then it
follows from \eqref{e39}, \eqref{e60} and \eqref{e61} that the
following equation holds \be\label{e62}u_M+T_du= u_{0,M},\ee where
$u_{0,M}$ is defined in \eqref{e47}. Using equations \eqref{e46} and
\eqref{e62}, we derive the following equality:
\be\begin{split}\label{e63}
(I_{d,M}+T_{d,M})(\tilde{u}_{(M)}-u_M)&=(I_{d,M}+T_{d,M})\tilde{u}_{(M)}-(I_{d,M}+T_{d,M})u_M\\
&=u_{0,M}-(I_{d,M}u_M+T_{d,M}u_M)\\
&=u_{0,M}-u_M-T_{d,M}u_M=T_du-T_{d,M}u_M,
\end{split}\ee where $\tilde{u}_{(M)}$ and $u_{0,M}$ are defined in \eqref{e47},
\be\label{e64}I_{d,M}v=v,\quad \forall v=\left(
                                           \begin{array}{c}
                                             v_1 \\
                                             v_2 \\
                                             \vdots \\
                                             v_M \\
                                           \end{array}
                                         \right)\in \C^M.\ee
Using relation \eqref{e63}, one gets \be\label{e65}
\tilde{u}_{(M)}-u_M=(I_{d,M}+T_{d,M})^{-1}(T_du-T_{d,M}u_M). \ee

\begin{lem}\label{lem32} Let Assumption A) hold (see Section 2 below
\eqref{e25}).
Suppose $u(x)$ solves \eqref{e39} where $p(x)\in C^1$. Then \be\label{e66}
\|\tilde{u}_{(M)}-u_M\|_{\C^M}=O(1/M^{2/3}) \text{ as } M\to
\infty,\ee where $u_M$ and $\tilde{u}_{(M)}$ are defined in
\eqref{e61} and \eqref{e47}, respectively.
\end{lem}
\begin{proof}
By \eqref{e65} and estimate \eqref{e52} we obtain
\be\begin{split}\label{e67} \|\tilde{u}_{(M)}-u_M\|_{\C^M}&\leq
\|(I_{d,M}+T_{d,M})^{-1}\|\|T_du-T_{d,M}u_M\|_{\C^M}\\
&\leq c_1\|T_du-T_{d,M}u_M\|_{\C^M},\end{split}\ee where $T_{d,M}$
and $T_d$ are defined in \eqref{e48} and \eqref{e60}, respectively.
Using the identity
$u(y)-u(x_j)=u(y)-u(x_j)-\mathcal{D}u(x_j)(y-x_j)+\mathcal{D}u(x_j)(y-x_j)$
and applying the triangle inequality, we get the estimate
\be\begin{split}\label{e68} \|T_du-T_{d,M}u_M\|_{\C^M}&=\max_{1\le
l\leq
M}\left|\sum_{j=1}^M\int_{D_j}g(x_l,y)p(y)(u(y)-u(x_j))dy\right|\\
&\leq \max_{1\le l\leq
M}\left|\int_{D_l}g(x_l,y)p(y)(u(y)-u(x_l))dy\right|\\
&+\max_{1\le l\leq
M}\left|\sum_{j=1,j\neq l}^M\int_{D_j}g(x_l,y)p(y)(u(y)-u(x_j))dy\right|\\
&\leq2\|p\|_\infty \|u\|_\infty\max_{1\le l\leq
M}\int_{D_l}|g(x_l,y)|dy\\
&+\max_{1\le l\leq M}\left|\sum_{j=1,j\neq
l}^M\int_{D_j}g(x_l,y)p(y)(u(y)-u(x_j))dy\right|\\
&\leq
2\|p\|_\infty\|u\|_\infty\int_{B_{\sqrt{3}/(2mP)}(x_l)}\frac{1}{4\pi|x_l-y|}dy\\
&+\max_{1\le l\leq M}\left|\sum_{j=1,j\neq
l}^M\int_{D_j}g(x_l,y)p(y)(u(y)-u(x_j))dy\right|\\
&=2\|p\|_\infty\|u\|_\infty\int_0^{\sqrt{3}/(2mP)}rdr\\
&+\max_{1\le l\leq M}\left|\sum_{j=1,j\neq
l}^M\int_{D_j}g(x_l,y)p(y)(u(y)-u(x_j))dy\right|\\
&\leq \frac{3\|p\|_\infty\|u\|_\infty}{2(mP)^2}+I_1+I_2,
\end{split}\ee where \be\label{e69} I_1:=\max_{1\le l\leq
M}\left|\sum_{j=1,j\neq
l}^M\int_{D_j}g(x_l,y)p(y)(u(y)-u(x_j)-\mathcal{D}u(x_j)(y-x_j))dy\right|
\ee and \be\label{e70} I_2:=\max_{1\le l\leq M}\left|\sum_{j=1,j\neq
l}^M\int_{D_j}g(x_l,y)p(y)\mathcal{D}u(x_j)(y-x_j)dy\right|. \ee Let
us derive an estimate for $I_1$. Using the Taylor expansion, we get
\be\label{e71} I_1\leq \|p\|_\infty\max_{1\le l\leq
M}\sum_{j=1,j\neq l}^M\int_{D_j}|g(x_l,y)|\sup_{0\leq s\leq
1}|\mathcal{D}^2u(sy+(1-s)x_j)||y-x_j|^2dy .\ee Since $p\in C^1(D)$,
$u\in C^2(D)$, $\int_D|g(x,y)|dy<\infty$ and $|y-x_j|\leq
\frac{\sqrt{3}}{2mP}$ for $y\in D_j$, it follows from \eqref{e71}
that \be\label{e72} I_1=O(1/(mP)^2)=O(1/M^{2/3}), \text{ as }M\to
\infty. \ee
Estimate of $I_2$ is obtained as follows. Since
$x_j$ is the center of the cube $D_j$, it follows that
\be\label{e73}
\int_{D_j}g(x_l,x_j)p(x_j)\mathcal{D}u(x_j)(y-x_j)dy=0,\quad
j=1,2,\hdots,M. \ee Therefore, using \eqref{e73}, $I_2$ can be
rewritten as follows: \be\label{e74}I_2=\max_{1\le l\leq
M}\left|\sum_{j=1,j\neq
l}^M\int_{D_j}(g(x_l,y)p(y)-g(x_l,x_j)p(x_j))\mathcal{D}u(x_j)(y-x_j)dy\right|.
\ee Let \be\label{e75} g_l(y):=g(x_l,y),\quad
(g_lp)(y)=g_l(y)p(y),\quad l=1,2,\hdots,M. \ee Then the formulas
\be\begin{split}\label{e76}
|(g_lp)(y)-(g_lp)(x_j)|&=\left|\int_0^1\frac{\partial}{\partial
t}(g_lp)(ty+(1-t)x_j) dt\right| \\
&\leq \sup_{0\leq t\leq
1}\left|\mathcal{D}_y(g_lp)(ty+(1-t)x_j)\right||y-x_j|\end{split}
\ee and
$\mathcal{D}_y(g_lp)(y)=p(y)\mathcal{D}_yg_l(y)+g_l(y)\mathcal{D}p(y)$,
yield the following estimate: \be\begin{split}\label{e77} I_2&\leq
\frac{\sqrt{3}\|\mathcal{D}u\|_\infty}{2mP}\max_{1\leq
l\leq M}\|p\|_\infty\int_D\sup_{0\leq t\leq 1}|\mathcal{D}_yg_l(ty+(1-t)x_j)||y-x_j|dy\\
&+\frac{\sqrt{3}\|\mathcal{D}u\|_\infty}{2mP}\max_{1\leq
l\leq M}\|\mathcal{D}p\|_\infty\int_D\sup_{0\leq t\leq 1}|g_l(ty+(1-t)x_j)||y-x_j|dy\\
&\leq \frac{c(k)\|\mathcal{D}u\|_\infty}{(mP)^2}\max_{1\leq l\leq
M}\|p\|_\infty\int_D\left(\frac{1}{4\pi|x_l-y|}+\frac{1}{4\pi|x_l-y|^2}\right)dy\\
&+\frac{\tilde{c}\|\mathcal{D}u\|_\infty}{(mP)^2}\max_{1\leq l\leq
M}\|\mathcal{D}p\|_\infty\int_D\frac{1}{4\pi|x_l-y|}dy=O(1/(mP)^2)=O(1/M^{2/3}),\end{split}\ee
where $\tilde{c}>0$ is a constant and $c(k)$ is a constant depending
on the wave number $k$. Here the estimates $|y-x_j|\leq
\sqrt{3}/(2mP)$ for $y\in D_j$,
$\int_D\frac{1}{4\pi|x_l-y|^\beta}dy<\infty$ for $\beta<3$, and
$|x_l-y|\leq 2|x_l-s|$ for $y\in D_j$, $j\neq l$, $s=tx_j+(1-t)y$,
$t\in [0,1]$, were used. The relation \eqref{e66} follows
from \eqref{e67}, \eqref{e68}, \eqref{e72} and \eqref{e77}.\\
\lemref{lem32} is proved.
\end{proof}

\begin{lem}\label{lem33} Let the Assumption A) hold. Consider
the linear algebraic system for the unknowns $u_e(x_l)$:
\be\label{e78} u_e(x_l)=u_0(x_l)-4\pi\sum_{j=1,j\neq
l}^{M}g(x_l,x_j)h(x_j)a^{2-\kappa}_{mP}u_e(x_j),\quad l=1,2,\hdots
M, \ee where $p(x)=4\pi h(x)N(x)\in C^2(D)$, $N(x)=1/\gamma^3,$
$M=(mP)^3$, and $g(x,y)$ is defined in \eqref{e40}. Then
\be\label{e79}
\|\tilde{u}_{(M)}-u_{e,M}\|_{\C^M}=O\left(\frac{\log M}
{M^{2/3}}+|1-\gamma^3Ma_{mP}^{2-\kappa}|\right) \text{ as }M\to
\infty, \ee where $\tilde{u}_{(M)}$ is defined in \eqref{e47},
 \be\label{e80}u_{e,M}:=\left(
                         \begin{array}{c}
                           u_e(x_1) \\
                           u_e(x_2) \\
                           \vdots \\
                           u_e(x_M) \\
                         \end{array}
                       \right)\in \C^M,\ee
and $u_e(x_j)$, $j=1,2,\hdots,M$, solve system \eqref{e78}.
\end{lem}
\begin{proof}
Let us rewrite \eqref{e78} as \be\begin{split}\label{e81}
u_e(x_l)&=u_0(x_l)-\sum_{j=1,j\neq l}^{M}g(x_l,x_j)p(x_j)\frac{a^{2-\kappa}_{mP}}{N(x_j)|D_j|}u_e(x_j)|D_j|\\
&=u_0(x_l)-(T_eu_{e,M})_l,\quad l=1,2,\hdots,M,
\end{split}\ee where $p(x)=4\pi h(x)N(x)$, $|D_j|=1/(mP)^3$ is the volume of the cube $D_j$,
$N(x)=1/\gamma^3$ and \be\label{e82} (T_ev)_l:=\sum_{j=1,j\neq
l}^{M}g(x_l,x_j)p(x_j)\left(\gamma
mPa^{(2-\kappa)/3}_{mP}\right)^3|D_j|v_j,\ v=\left(
                         \begin{array}{c}
                           v_1 \\
                           v_2 \\
                           \vdots \\
                           v_M \\
                         \end{array}
                       \right)\in \C^M,\ee $l=1,2,\hdots, M$.\\ Let us derive an estimate for $\|\tilde{u}_{(M)}-u_{e,M}\|_{\C^M}$. Using
equations \eqref{e46} and \eqref{e81}, we obtain
\be\begin{split}\label{e83}
(I_{d,M}+T_{d,M})(\tilde{u}_{(M)}-u_{e,M})&=
(I_{d,M}+T_{d,M})\tilde{u}_{(M)}-(I_{d,M}+T_{d,M})u_{e,M}\\
&=u_{0,M}-u_{e,M}-T_{d,M}u_{e,M}=T_eu_{e,M}-T_{d,M}u_{e,M},\end{split}\ee
where $\tilde{u}_{(M)}$ and $u_{0,M}$ are defined in \eqref{e47},
$u_{e,M},$ $T_{d,M}$ and $T_e$ are defined in \eqref{e80},
\eqref{e48} and \eqref{e82}, respectively. Relation \eqref{e83}
implies \be\begin{split}\label{e84}
\|\tilde{u}_{(M)}-u_{e,M}\|_{\C^M}&\leq
\|(I_{d,M}+T_{d,M})^{-1}\|\|T_eu_{e,M}-T_{d,M}u_{e,M}\|_{\C^M}\\
&\leq c_1\|T_eu_{e,M}-T_{d,M}u_{e,M}\|_{\C^M},\end{split}\ee where
estimate \eqref{e52} was used.

Let us derive an estimate for
$\|T_eu_{e,M}-T_{d,M}u_{e,M}\|_{\C^M}$. Using definitions
\eqref{e48} and \eqref{e82}, and applying the triangle inequality,
one gets
\be\label{e85}\begin{split} &\|(T_e-T_{d,M})u_{e,M}\|_{\C^M}\\
&=\max_{1\leq l\leq M}\left|\sum_{j=1,j\neq l}^Mg_l(x_j)p_{a_{mP}}(x_j)u_e(x_j)|D_j|-\sum_{j=1}^M\int_{D_j}g_l(y)p(y)dyu_e(x_j) \right|\\
&=\max_{1\leq l\leq M}\left|\sum_{j=1,j\neq l}^M\int_{D_j}\left(g_l(x_j)p_{a_{mP}}(x_j)-g_l(y)p(y)\right)dyu_e(x_j)-\int_{D_l}g_l(y)p(y)dyu_e(x_l)\right|\\
&\leq \max_{1\leq l\leq
M}\left|\int_{D_l}g_l(y)p(y)u_e(x_l)dy\right|+\max_{1\leq l\leq
M}\left|\sum_{j=1,j\neq
l}^M\int_{D_j}g_l(y)(p(y)-p(x_j))u_e(x_j)dy\right| \\
&+\max_{1\leq l\leq M}\left|\sum_{j=1,j\neq
l}^M\int_{D_j}g_l(y)(p(x_j)-p_{a_{mP}}(x_j))u_e(x_j)dy\right|\\
&+ \max_{1\leq l\leq M}\left|\sum_{j=1,j\neq
l}^Mp_{a_{mP}}(x_j)\int_{D_j}(g_l(y)-g_l(x_j))u_e(x_j)dy\right|\\
&\leq \max_{1\leq l\leq M}(J_0(l)+J_1(l)+J_2(l)+J_3(l)),
\end{split}\ee where $g_l(y):=g(x_l,y)$, $ p_{a_{mP}}(x):=p(x)\left(\gamma
mPa^{(2-\kappa)/3}_{mP}\right)^3$, \be\label{e86}
J_0(l):=\int_{D_l}\left|g(x_l,y)p(y)u_e(x_l)\right|dy, \ee
\be\label{e87} J_1(l):=\sum_{j=1,j\neq
l}^M\left|\int_{D_j}g(x_l,y)(p(y)-p(x_j))u_e(x_j)dy\right|, \ee
\be\label{e88} J_2(l):=\sum_{j=1,j\neq
l}^M\int_{D_j}\left|g(x_l,y)(p(x_j)-p_{a_{mP}}(x_j))u_e(x_j)\right|dy,\ee
and \be\label{e89} J_3(l):=\sum_{j=1,j\neq
l}^M|p_{a_{mP}}(x_j)|\left|\int_{D_j}(g(x_l,y)-g(x_l,x_j))u_e(x_j)dy\right|.\ee

Using the estimate $|x_l-y|\leq \sqrt{3}/(2mP)$ for $y\in D_l$, one
gets the following estimate of $J_0(l)$: \be\begin{split}\label{e90}
J_0(l)&\leq \|p\|_\infty\|u_e\|_{\C^M}\int_{D_l}|g(x_l,y)|dy\\
&\leq
\left(\int_{B_{\sqrt{3}/(2mP)}(x_l)}\frac{1}{4\pi|x_l-y|}dy\right)\|p\|_\infty\|u_e\|_{\C^M}\\
&=\left(\int_0^{\sqrt{3}/(2mP)}rdr\right)\|p\|_\infty\|u_e\|_{\C^M}\\
&=\frac{3\|p\|_\infty\|u_e\|_{\C^M} }{2(2mP)^2}=O(1/M^{2/3}),
\end{split}\ee where $B_a(x)$ is defined in \eqref{e19}.

Let us estimate $J_1(l)$. Using the identity \be\label{e91}
p(y)-p(x_j)=p(y)-p(x_j)-\mathcal{D}p(x_j)\cdot
(y-x_j)+\mathcal{D}p(x_j)\cdot (y-x_j) \ee in \eqref{e87} and
applying the triangle inequality, one obtains
\be\label{e92} J_1(l)\leq
J_{1,1}+J_{1,2}, \ee where \be\label{e93} J_{1,1}:=\sum_{j=1,j\neq
l}^M\left|\int_{D_j}g(x_l,y)[p(y)-p(x_j)-\mathcal{D}p(x_j)\cdot
(y-x_j)]u_e(x_j)dy\right|, \ee and \be\label{e94}
J_{1,2}:=\sum_{j=1,j\neq
l}^M\left|\int_{D_j}g(x_l,y)\mathcal{D}p(x_j)\cdot
(y-x_j)u_e(x_j)dy\right|. \ee To get an estimate for $J_{1,1}$, we
apply the Taylor expansion of $p(x)$ and get
\be\begin{split}\label{e95} J_{1,1}&\leq
\frac{\|u_e\|_{\C^M}}{2}\sum_{j=1,j\neq l}^M
\int_{D_j}|g(x_l,y)|\sup_{0\leq t\leq 1}|\mathcal{D}^2p(ty+(1-t)x_j)||y-x_j|^2dy\\
&\leq
\frac{3\|\mathcal{D}^2p\|_\infty\|u_e\|_{\C^M}}{8(mP)^2}\sum_{j=1,j\neq
l}^M\int_{D_j}|g(x_l,y)|dy\\
&=O(1/(mP)^2)=O(1/M^{2/3})\ \text{ as }M\to \infty,
\end{split}\ee
where $B_a(x)$ is defined in \eqref{e19}, and the estimate
$|y-x_j|\leq \sqrt{3}/(2mP),\ y\in D_j,$ was used.

Using the identity
\be\label{e96} \int_{D_j}g(x_l,x_j)\mathcal{D}p(x_j)(y-x_j)u_e(x_j)
dy=0,\quad j=1,2,\hdots, M,\ee
one derives the following estimate of
$J_{1,2}$:
\be\begin{split}\label{e97} J_{1,2}&=\sum_{j=1,j\neq
l}^M\left|\int_{D_j}(g(x_l,y)-g(x_l,x_j))\mathcal{D}p(x_j)\cdot
(y-x_j)u_e(x_j)dy\right|\\
&\leq \|u_e\|_{\C^M}\sum_{j=1,j\neq l}^M\int_{D_j}\left|(g(x_l,y)-g(x_l,x_j))\mathcal{D}p(x_j)\cdot(y-x_j)\right|dy\\
&\leq \frac{\sqrt{3}\|p\|_\infty\|u_e\|_{\C^M}}{2mP}\sum_{j=1,j\neq l}^M\int_{D_j}|g(x_l,y)-g(x_l,x_j)|dy\\
&\leq \frac{\sqrt{3}\|p\|_\infty\|u_e\|_{\C^M}}{2mP}\sum_{j=1,j\neq l}^M\int_{D_j}\sup_{0\leq t\leq 1}|\mathcal{D}g(x_l,ty+(1-t)x_j)||y-x_j|dy\\
&\leq \frac{c(k)\|p\|_\infty\|u_e\|_{\C^M}}{(mP)^2}\sum_{j=1,j\neq l}^M\int_{D_j}\sup_{0\leq t\leq 1}\frac{1}{4\pi|x_l-ty-(1-t)x_j|^2}dy\\
&+\frac{c(k)\|p\|_\infty\|u_e\|_{\C^M}}{(mP)^2}\sum_{j=1,j\neq l}^M\int_{D_j}\sup_{0\leq t\leq 1}\frac{1}{4\pi|x_l-ty-(1-t)x_j|}dy\\
&\leq \frac{c(k)\|p\|_\infty\|u_e\|_{\C^M}}{(mP)^2}\int_{B_{\sqrt{3}}(x_l)}\left(\frac{1}{4\pi|x_l-y|}+\frac{1}{4\pi|x_l-y|^2}\right)dy\\
&=O(1/(mP)^2)=O(1/M^{2/3}) \text{ as }M\to \infty,\end{split}\ee
where $c(k)$ is a constant depending on the wave number $k$ and
$B_a(x)$ is defined in \eqref{e19}. Here the estimates $|y-x_j|\leq
\sqrt{3}/(2mP)$ for $y\in D_j$,
$\int_D\frac{1}{4\pi|x_l-y|^\beta}dy<\infty$ for $\beta<3$, and
$|x_l-y|\leq 2|x_l-s|$ for $y\in D_j$, $j\neq l$, $s=tx_j+(1-t)y$,
$t\in [0,1]$, were used. Applying estimates \eqref{e95} and
\eqref{e97} to \eqref{e92}, we get \be\label{e98}
J_1(l)=O(1/M^{2/3}), \text{ as }M\to \infty. \ee

Let us derive an estimate for $J_2(l)$. From \eqref{e88} and the
definition $p_{a_{mP}}(x)=p(x)(\gamma mPa^{(2-\kappa)/3})^3$  we get
\be\begin{split}\label{e99}
J_2(l)&\leq \|u_e\|_{\C^M}\sum_{j=1,j\neq l}^M\int_{D_j}|g(x_l,y)||p(x_j)||1-(\gamma mPa^{(2-\kappa)/3})^3|dy\\
&\leq \|u_e\|_{\C^M}\|p\|_\infty |1-(\gamma mPa^{(2-\kappa)/3})^3|\sum_{j=1,j\neq l}^M\int_{D_j}|g(x_l,y)|dy\\
&\leq \|u_e\|_{\C^M}\|p\|_\infty |1-(\gamma
mPa^{(2-\kappa)/3})^3|\int_{B_{\sqrt{3}}(x_l)}\frac{1}{4\pi|x_l-y|}dy\\
&=\left(\int_0^{\sqrt{3}}rdr\right)\|u_e\|_{\C^M}\|p\|_\infty
|1-(\gamma mPa^{(2-\kappa)/3})^3|\\
&= \frac{3}{2}\|u_e\|_{\C^M}\|p\|_\infty |1-\gamma^3
Ma_{mP}^{(2-\kappa)}|=O(|1-\gamma^3 Ma_{mP}^{(2-\kappa)}|),
\end{split}\ee where $B_a(x)$ is defined in \eqref{e19}. Using the
relation $\lim_{M\to \infty}\gamma^3 Ma^{(2-\kappa)}=1$, one gets
$J_2(l)\to 0$ as $M\to \infty$. Estimate of $J_3(l)$ is derived
as follows. Using the identity
\be\label{e100}
\int_{D_j}\mathcal{D}g(x_l,x_j)(y-x_j)u_e(x_j)dy=0 ,\quad
j=1,2,\hdots M,\ee one gets the following
estimate:\be\begin{split}\label{e101}
J_3(l)&=\sum_{j=1,j\neq l}^M|p_{a_{mP}}(x_j)|\left|u_e(x_j)\int_{D_j}\left[g(x_l,y)-g(x_l,x_j)-\mathcal{D}g(x_l,x_j)\cdot(y-x_j)\right]dy\right|\\
&\leq \frac{\|p\|_\infty\gamma^3Ma_{mP}^{2-\kappa}}{2}\sum_{j=1,j\neq l}^M\int_{D_j}\sup_{0\leq t\leq 1}|\mathcal{D}^2g(x_l,ty+(1-t)x_j)||y-x_j|^2dy\\
&\leq \frac{3c_M}{8(mP)^2}\sum_{j=1,j\neq l}^M\int_{D_j}\sup_{0\leq t\leq 1}|\mathcal{D}^2g(x_l,ty+(1-t)x_j)|dy\\
&\leq \frac{c(k)c_M}{(mP)^2}\sum_{j=1,j\neq l}^M\int_{D_j}\sup_{0\leq t\leq 1}\frac{1}{4\pi|x_l-ty-(1-t)x_j|}dy\\
&+\frac{c(k)c_M}{(mP)^2}\sum_{j=1,j\neq l}^M\int_{D_j}\sup_{0\leq t\leq 1}\frac{1}{4\pi|x_l-ty-(1-t)x_j|^2}dy\\
&+ \frac{c(k)c_M}{(mP)^2}\sum_{j=1,j\neq l}^M\int_{D_j}\sup_{0\leq t\leq 1}\frac{1}{4\pi|x_l-ty-(1-t)x_j|^3}dy\\
&\leq \frac{c(k)c_M}{(mP)^2}\int_{1/(2mP)<|x_l-y|<\sqrt{3}}\left(\frac{1}{4\pi|x_l-y|}+\frac{1}{4\pi|x_l-y|^2}+\frac{1}{4\pi|x_l-y|^3}\right)dy\\
&\leq\frac{c(k)c_M}{(mP)^2}\int_{1/(2mP)}^{\sqrt{3}}\left(r+1+\frac{1}{r}\right)dr\\
&\leq
\frac{2c(k)c_M}{(mP)^2}\left[1+\log(\sqrt{3})-\log\left(1/(2mP)\right)\right]\\
&=\frac{2c(k)c_M}{M^{2/3}}\left[1+\log(\sqrt{3})-\log\left(1/(2M^{1/3})\right)\right]=O\left(\frac{\log
M}{M^{2/3}}\right),
\end{split}\ee
where $c_M:=\|p\|_\infty\gamma^3Ma_{mP}^{2-\kappa}$, $c(k)$ is a
constant depending on the wave number $k$. Here the estimates $|y-x_j|\leq
\sqrt{3}/(2mP)$ for $y\in D_j$, and $|x_l-y|\leq 2|x_l-s|$ for $y\in
D_j$, $j\neq l$, $s=tx_j+(1-t)y$, $t\in [0,1]$, were used.
Using estimates \eqref{e95}, \eqref{e97}, \eqref{e99} and \eqref{e101}, one gets relation \eqref{e79}.\\
\lemref{lem33} is proved.
\end{proof}
The following theorem is a consequence of \lemref{lem32} and
\lemref{lem33}.
\begin{thm}\label{thm34}
Suppose that the assumptions of \lemref{lem32} and \lemref{lem33}
hold. Then \be\label{e102}
\|u_M-u_{e,M}\|_{\C^M}=O\left(\frac{\log M }{M^{2/3}}+
|1-\gamma^3Ma_{mP}^{2-\kappa}|\right)
\text{ as }M\to \infty, \ee where $u_M$ and $u_{e,M}$ are defined in
\eqref{e61} and \eqref{e80}, respectively.
\end{thm}
To get the rate of convergence \eqref{e102} we have assumed that
$p(x)\in C^2(D)$. If $p(x)\in C(D)$ then the rate given
in \thmref{thm34} is no longer valid. The rate of
$\|u_M-u_{e,M}\|_{\C^M}$ when $p(x)\in C(D)$ is given in the
following theorem.
\begin{thm}\label{thm35}
Let Assumption A) hold and $p\in C(D)$ satisfies \be |p(x)-p(y)|\leq
\omega_p(|x-y|),\quad \forall x,y\in D,\ee where $\omega_p$ is the
modulus of continuity of the function $p(x)$.Then
\be\label{e104}\begin{split}
\|u_M-u_{e,M}\|_{\C^M}&=O\left(\frac{\log M }{M^{1/3}}+
|1-\gamma^3Ma_{mP}^{2-\kappa}|+ \omega_p(1/M^{1/3})\right),
\end{split}\ee
where $u_M$ and $u_{e,M}$ are defined in \eqref{e61}
and \eqref{e80}, respectively.
\end{thm}
\begin{proof}

We have \be\label{e105} \|u_M-u_{e,M}\|_{\C^M}\leq
\|u_M-\tilde{u}_M\|_{\C^M}+\|\tilde{u}_M-u_{e,M}\|_{\C^M}. \ee Let
us estimate $\|u_M-\tilde{u}_M\|_{\C^M}$. From \eqref{e67} we have
\be \|u_M-\tilde{u}_M\|_{\C^M}\leq c_1
\|T_du-T_{d,M}u_M\|_{\C^M},\ee where $c_1$ is defined in
\eqref{e52}. Using the similar steps given in \eqref{e68} we get the
following estimate for $\|T_du-T_{d,M}u_M\|_{\C^M}$ :
\be\label{e107}\|T_du-T_{d,M}u_M\|_{\C^M}\leq
\frac{3\|p\|_\infty\|u\|_\infty}{2(mP)^2}+I_1+I_2, \ee where $I_1$
and $I_2$ are defined in \eqref{e69} and \eqref{e70}, respectively.
It is shown in \eqref{e72} that $I_1=O(1/M^{2/3}).$ Since $p(x)\in
C(D)$, the steps \eqref{e74}-\eqref{e77} are no longer valid. The
estimate of $I_2$ can be derived as follows. Since $p\in C(D)$ and
$u\in C^1(D)$, it follows from \eqref{e70} that \be\begin{split}
I_2&\leq \|p\|_\infty\|\mathcal{D}u\|_\infty\max_{1\leq l\leq
M}\sum_{j=1,\neq l}^M\int_{D_j}|g(x_l,y)||x_l-y|dy\\
&\leq
\frac{\sqrt{3}\|p\|_\infty\|\mathcal{D}u\|_\infty}{2mP}\max_{1\leq
l\leq
M}\int_{B_{\sqrt{3}}(x_l)}\frac{1}{4\pi|x_l-y|}dy=O(1/(mP))
=O(1/M^{1/3}).\end{split}\ee
This together with \eqref{e107} and $I_1=O(1/M^{2/3})$ yield
\be\label{e109} \|u_M-\tilde{u}_M\|_{\C^M}=O(1/M^{1/3}). \ee

Let us find an estimate for $\|\tilde{u}_M-u_{e,M}\|_{\C^M}$. From
\eqref{e84} we have \be \|\tilde{u}_M-u_{e,M}\|_{\C^M}\leq
c_1\|T_eu_{e,M}-T_{d,M}u_{e,M}\|_{\C^M},\ee where $c_1$ is defined
in \eqref{e52}. By definitions \eqref{e48} and \eqref{e82}, and use
the triangle inequality, we get \be
\|T_eu_{e,M}-T_{d,M}u_{e,M}\|_{\C^M}\leq J_1+J_2, \ee where \be
 J_1:=\max_{1\leq
l\leq M}\left|\int_{D_l}g(x_l,y)p(y)dyu_e(x_l)dy \right| \ee and \be
J_2:=\max_{1\leq l\leq M}\left|\sum_{j=1,j\neq
l}^M\int_{D_j}\left(g(x_l,x_j)p_{a_{mP}}(x_j)-g(x_l,y)p(y)\right)u_e(x_j)dy\right|,\ee
$p_{a_{mP}}:=p(x)(\gamma mPa^{(2-\kappa)/3})^3$. It is proved in
\eqref{e90} that $J_1=O(1/M^{2/3}).$ The estimate of $J_2$ is
derived as follows. By the triangle inequality we obtain \be J_2\leq
J_{2,1}+J_{2,2}, \ee where \be J_{2,1}:=\max_{1\leq l\leq
M}\left|\sum_{j=1,j\neq
l}^M\int_{D_j}(g(x_l,x_j)-g(x_l,y))p_{a_{mP}}(x_j)u_e(x_j)dy\right|
\ee and \be J_{2,2}:=\max_{1\leq l\leq
M}\|u_e\|_{\C^M}\sum_{j=1,j\neq
l}^M\int_{D_j}|g(x_l,y)||p_{a_{mP}}(x_j)-p(y)|dy .\ee It is proved
in \eqref{e101} that \be\label{e117} J_{2,1}=O\left(\frac{\log M
}{M^{2/3}}\right). \ee To estimate $J_{2,2}$, we apply the triangle
inequality and get \be\begin{split}\label{e118} J_{2,2}&\leq
\|u_e\|_{\C^M}\max_{1\leq l\leq M}\sum_{j=1,j\neq
l}^M\int_{D_j}|g(x_l,y)||p_{a_{mP}}(x_j)-p(x_j)|dy\\
&+\|u_e\|_{\C^M}\max_{1\leq l\leq M}\sum_{j=1,j\neq
l}^M\int_{D_j}|g(x_l,y)||p(x_j)-p(y)|dy\\
&\leq |\gamma^3Ma_{mP}^{2-\kappa}-1|\|p\|_\infty\max_{1\leq l\leq
M}\|u_e\|_{\C^M}\sum_{j=1,j\neq
l}^M\int_{D_j}|g(x_l,y)|dy\\
&+\max_{j}\sup_{y\in D_j}\omega_p(|x_j-y|)\|u_e\|_{\C^M}\max_{1\leq
l\leq M}\sum_{j=1,j\neq l}^M\int_{D_j}|g(x_l,y)|dy\\
&=O\left(|\gamma^3Ma_{mP}^{2-\kappa}-1|+ \omega_p(1/M^{1/3})\right),
\end{split}\ee where $\omega_p $ is the modulus of continuity of
$p(x)$. This together with $J_1=O(1/M^{2/3})$ and \eqref{e117} yield
\be\label{e119}\begin{split}
\|\tilde{u}_M-u_{e,M}\|_{\C^M}&=O\left(\frac{ \log M }{M^{2/3}}+
|\gamma^3Ma_{mP}^{2-\kappa}-1|+\omega_p(1/M^{1/3})\right).
\end{split}\ee
Relation \eqref{e104}
follows from \eqref{e105}, \eqref{e109} and
\eqref{e119}.\\
\thmref{thm35} is proved.

\end{proof}

As we mentioned in the introduction, the main goal is to develop an
algorithm for obtaining the minimal number of the embedded small
balls which generate a material whose refraction coefficient differs
from the desired one by not more than a desired small quantity. Let
us derive an approximation of the desired refraction coefficient
$n^2(x)$ generated by the embedded small balls. We rewrite the sum
in \eqref{e38} as \be\begin{split}\label{e120}
\sum_{j=1}^Mg(x,x_j)p_{a_{mP}}(x_j)u(x_j)|D_j|,
\end{split}\ee where $|x-x_j|>a_{mP}$, $j=1,2,\hdots,M$, $|D_j|=1/(mP)^3$ is the volume of the cube $D_j$, and \be\label{e121}
p_{a_{mP}}(x):=4\pi h(x)N(x)(\gamma mPa^{(2-\kappa)/3})^3,\quad
N(x)=1/\gamma^3. \ee  Since $(\gamma mPa^{(2-\kappa)/3})^3\to 1$ as
$m\to \infty$, it follows that \eqref{e120} is a Riemannian sum for
the integral $\int_D g(x,y)p(y)u(y)dy,$ where $p(x)=4\pi h(x)N(x)$.
 This motivates us to define the following approximation of the refraction coefficient $n^2(x)$:
\be\label{e122} n_{a_{mP}}^2(x):=n_0^2(x)-k^{-2}p_{a_{mP}}(x), \ee
where $p_{a_{mP}}$ is defined in \eqref{e121}. We are interested in
finding the largest radius $a_{mP}$ (or the smallest $M=(mP)^3$)
such that \be\label{e123} e(M):=\max_{1\leq l\leq
M}|n^2(x_l)-n^2_{a_{mP}}(x_l)|\leq \epsilon/k^2:=\epsilon(k), \ee
where $k$ is the wave number, $\epsilon>0$ is a given small quantity
and $n_{a_{mP}}^2(x)$ is defined in \eqref{e122}.

An estimate of the error $e(M)$, defined in \eqref{e123}, is given
in the following theorem.

\begin{thm}\label{thm36}
Suppose Assumption A) holds and $N(x)=1/\gamma^3$. Then
\be\label{e124}
 \max_{1\leq l\leq (mP)^3}|n^2(x_l)-n^2_{a_{mP}}(x_l)|\leq k^{-2}\|p\|_\infty\left|1-(\gamma mPa_{mP}^{(2-\kappa)/3})^3\right|,
 \ee where $x_l$ is the center of the $l$-th small ball, $p(x)$ is
 defined in \eqref{e41}, $n^2(x)=n_0^2(x)-k^{-2}p(x),$  and $n^2_{a_{mP}}$ is defined in
\eqref{e122}. Consequently,
 \be\label{e125}
 \lim_{m\to \infty} \max_{1\leq l\leq (mP)^3}|n^2(x_l)-n_{a_{mP}}^2(x_l)|=0.
 \ee
\end{thm}
\begin{proof}
Let \bee I_l:=|n(x_l)-n^2_{a_{mP}}(x_l)|.\eee Then
\be\label{e126}\begin{split}
 I_l&= k^{-2}|p(x_l)-p(x_l)[\gamma mPa_{mP}^{(2-\kappa)/3}]^3|\leq
k^{-2}|p(x_l)||1-[\gamma mPa_{mP}^{(2-\kappa)/3}]^3|\\
&\leq k^{-2}\|p\|_\infty|1-[\gamma mPa_{mP}^{(2-\kappa)/3}]^3|.
\end{split}\ee
 This together with relation \eqref{e22} yield \eqref{e125}.\\
 \thmref{thm36} is proved.

\end{proof}

Using  \thmref{thm36} one can calculate the smallest $M$ satisfying
\eqref{e123} by the following algorithm:\\

\noindent \textbf{Algorithm}\\
\noindent {\bf Initializations}: Let the wave number $k$, the
constant $\epsilon>0$, $n_0^2(x)$ and $n^2(x)$ be given. Fix $P>1$,
$m=m_0:=1$, $\kappa\in(0,1)$, $\gamma > [1/(2P)]^{(\kappa+1)/3}$ and
$N(x)=1/\gamma^3$. Partition $D$ into $P^3$ cubes $\Omega_q$,
$D=\cup_{q=1}^{P^3}\overline{\Omega_q}$,
$\Omega_j\cap\Omega_i=\emptyset$ for $j\neq i$, where each cube
$\Omega_j$ has side length $1/P$.
\begin{enumerate}

\item[\textbf{Step 1.}] Solve the equation \be\label{e127}
\gamma a_{mP}^{(2-\kappa)/3}+a_{mP}-1/(mP)=0\ee for $a_{mP}.$
\item[\textbf{Step 2.}] Embed $m^3$ small balls of radius $a_{mP}$ in each
cube $\Omega_q$ so that Assumption A) holds.
\item[\textbf{Step 3.}] Compute
\bee p(x_l)=k^2(n^2_0(x_l)-n^2(x_l))\eee and \bee
p_{a_{mP}}(x_l)=p(x_l)[\gamma m P
a_{mP}^{(2-\kappa)/3}]^3,\qquad l=1,2,\hdots,(mP)^3, \eee where
$x_l$ is the center of the $l$-th small ball and $k$ is the wave
number.
\item[\textbf{Step 4.}] If $\max_{1\leq l\leq (mP)^3}|p(x_l)-p_{a_{mP}}(x_l)|>
\epsilon $, then set $m=m+1$ and go to \textbf{Step 1}.
Otherwise the number $M=(mP)^3$ is the smallest number of the
balls embedded in $D$ such that inequality \eqref{e123} holds,
and $a_{mP}$ is the radius of each embedded ball.
\end{enumerate}

\section{Numerical experiments}
In this section we give the results of the  numerical experiments.
Suppose the refraction coefficient of the original material in $D$
is $n_0^2(x)=1$ and the desired refraction coefficients are:
\begin{enumerate}
\item[Example 1.] $n^2(x)=5$,
\item[Example 2.]
$n^2(x)=5+\exp(-|x-x_0|^2/(2\sigma^2))/(\sqrt{2\pi}\sigma), $
where $x_0=(0.5,0.5,0.5)$ and $\sigma=\frac{\sqrt{3}}{2b P}$.
Here $b$ is the smallest number $m$ taken from Example 1.
\item[Example 3.] $n^2(x)=1+0.5\sin(x_1)$, where $x_1$ is the first
element of
the vector $x$,
\item[Example 4.] $n^2(x)=1+0.5\sin(100 x_1)$, where $x_1$ is the first
element of
the vector $x$. \end{enumerate}

By the recipe we choose \be h(x)=k^2[n_0^2(x)-n^2(x)]/(4\pi
N(x))=\gamma^3k^2(n_0^2(x)-n^2(x))/(4\pi ) ,\quad k>0. \ee Let us
take \be P=11,\quad \kappa=0.99,\quad
\gamma=10k[1/(2P)]^{(1+\kappa)/3},\quad m_0=1, \ee where $k\geq 1$
is the wave number and $m_0$ is the initial number of small balls
described in the algorithm. Here the parameters $P=11$ and $m_0=1$
are chosen so that the approximation error in \lemref{lem32} is at
most $c(k)10^{-4}$, where $c(k)$ is a constant depending on the wave
number $k$. We apply the algorithm given in Section 3 to get the
minimal total number of small balls embedded in the cube $D$ such
that inequality \eqref{e123} holds for various values of $\epsilon$,
where the quantity $\epsilon$ was defined in the Algorithm (see the
Initialization and Step 4 of the Algorithm).

The smallest number of the balls embedded in $D$ increases as
$\epsilon$ decreases.  The radius $a_{mP}$ and the
ratio $a_{mP}/d_{mP}$ decrease as $M$ increases, which agrees with
the theory. The results are shown in tables 1-4. In these tables we
define \be d_{mP}:=\min_{1\leq i,j \leq M, i\neq
j}\text{dist}(B_{a_{mP}}(x_i),B_{a_{mP}} (x_j)),\ee where $B_a(x)$
is defined in \eqref{e19}, and \be E:=\max_{1\leq l\leq
M}|n^2(x_l)-n_{a_{mP}}^2(x_l)|, \ee where $M$ is the smallest total
number of small balls embedded in the domain $D$, $a_{mP}$ is the
radius of the embedded small balls and $x_l$ is the center of the $l$-th
small ball. In Example 1 we choose a constant refraction coefficient
$n^2(x)$. For $k=1$ the total number of small balls $M$ increases by
$1.651\times 10^{5}$ when the error level $\epsilon$ is decreased by
$50\%$, while for $k=5$ the value of $M$ increases by $3.4609\times
10^{4}$ as the error level $\epsilon$ decreases by $50\%$, as shown
in Table 1.

In Example 2 we add a Gaussian function to the constant refraction
coefficient $n^2(x)$ considered in Example 1. For $k=1$ the value of $M$
increases significantly as the error level $\epsilon$ decreases by
$50\%$.
\begin{table}[htp]
\newcommand{\m}{\hphantom{$-$}}
\renewcommand{\tabcolsep}{.85pc} 
\renewcommand{\arraystretch}{1.2} 
\scriptsize{
\begin{tabular}{cccccc}
\hline
&\multicolumn{5}{c}{$k=1$}\\
\hline
$\epsilon$&$m$&$ M$&$ a_{mP}$&$  a_{mP}/d_{mP}$& $E$\\
\hline

$   5.000\times 10^{-1} $& $1$& $  1.331\times 10^{3}$& $  3.722\times 10^{-4}$ &$5.445\times 10^{-2}$&$  9.747\times 10^{-2}$ \\
$   5.000\times 10^{-3} $& $5$& $  1.664\times 10^{5}$& $  3.198\times 10^{-6}$ &$  1.098\times 10^{-2}$&$  4.219\times 10^{-3}$ \\

\hline \hline
&\multicolumn{5}{c}{$k=5$}\\
\hline
$\epsilon$&$m$&$ M$&$ a_{mP}$&$  a_{mP}/d_{mP}$& $E$\\
\hline

$   5.000\times 10^{-1} $& $1$& $  1.331\times 10^{3}$& $  3.200\times 10^{-6}$ &$  2.196\times 10^{-3}$&$  8.448\times 10^{-4}$ \\
$   5.000\times 10^{-3} $& $3$& $  3.594\times 10^{4}$& $  1.225\times 10^{-7}$ &$  7.320\times 10^{-4}$&$  9.701\times 10^{-5}$ \\
\hline \hline
\end{tabular}}
\caption{ Example 1 }
\end{table}
\begin{table}[htp]
\newcommand{\m}{\hphantom{$-$}}
\renewcommand{\tabcolsep}{.85pc} 
\renewcommand{\arraystretch}{1.2} 
\scriptsize{
\begin{tabular}{cccccc}
\hline
&\multicolumn{5}{c}{$k=1$}\\
\hline
$\epsilon$&$m$&$ M$&$ a_{mP}$&$  a_{mP}/d_{mP}$& $E$\\
\hline

$   5.000\times 10^{-1} $& $1$& $  1.331\times 10^{3}$& $ 3.722\times 10^{-4}$ &$  5.445\times 10^{-2}$&$  2.209\times 10^{-1}$ \\
$   5.000\times 10^{-3} $& $13$& $  2.924\times 10^{6}$& $  1.873\times 10^{-7}$ &$  4.223\times 10^{-3}$&$  4.715\times 10^{-3}$ \\
\hline \hline
&\multicolumn{5}{c}{$k=5$}\\
\hline
$\epsilon$&$m$&$ M$&$ a_{mP}$&$  a_{mP}/d_{mP}$& $E$\\
\hline

$   5.000\times 10^{-1} $& $1$& $  1.331\times 10^{3}$& $  3.200\times 10^{-6}$ &$  2.196\times 10^{-3}$&$  1.915\times 10^{-3}$ \\
$   5.000\times 10^{-3} $& $5$& $  1.664\times 10^{5}$& $  2.686\times 10^{-8}$ &$  4.392\times 10^{-4}$&$  1.702\times 10^{-4}$ \\
\hline \hline
\end{tabular}}
\caption{ Example 2 }
\end{table}

The refraction coefficients $n^2(x)$ considered in Examples 3 and 4 are
periodic. In Example 3 for $k=1$ the total number of the embedded
small particles $M$ increases by $9319$ as the error level
$\epsilon$ decreases by $50\%$. A similar increment of $M$ is
obtained for $k=5$. These results are shown in Table 3.

In Example 4 the angular frequency of the sine function is $100$
times the angular frequency of the sine function given in Example 3.
In this case we get significant increments of the value of $M$ as
the error level $\epsilon$ decreases by $50\%$ for the wave numbers
$k=1$ and $5$, see  Table 4.

\begin{table}[htp]
\newcommand{\m}{\hphantom{$-$}}
\renewcommand{\tabcolsep}{.85pc} 
\renewcommand{\arraystretch}{1.2} 
\scriptsize{
\begin{tabular}{cccccc}
\hline
&\multicolumn{5}{c}{$k=1$}\\
\hline
$\epsilon$&$m$&$ M$&$ a_{mP}$&$  a_{mP}/d_{mP}$& $E$\\
\hline
$   5.000\times 10^{-1} $& $1$& $  1.331\times 10^{3}$& $  3.722\times 10^{-4}$ &$  5.445\times 10^{-2}$&$  5.536\times 10^{-4}$ \\
$   5.000\times 10^{-4} $& $2$& $  1.065\times 10^{4}$& $  4.837\times 10^{-5}$ &$  2.739\times 10^{-2}$&$  7.239\times 10^{-5}$ \\
\hline \hline
&\multicolumn{5}{c}{$k=5$}\\
\hline
$\epsilon$&$m$&$ M$&$ a_{mP}$&$  a_{mP}/d_{mP}$& $E$\\
\hline
$   5.000\times 10^{-1} $& $1$& $  1.331\times 10^{3}$& $  3.200\times 10^{-6}$ &$  2.196\times 10^{-3}$&$ 4.798\times 10^{-6}$ \\
$   5.000\times 10^{-5} $& $2$& $ 1.065\times 10^{4}$& $  4.084\times 10^{-7}$ &$  1.098\times 10^{-3}$&$  6.126\times 10^{-7}$ \\
\hline \hline
\end{tabular}}
\caption{ Example 3 }
\end{table}

\begin{table}[htp]
\newcommand{\m}{\hphantom{$-$}}
\renewcommand{\tabcolsep}{.85pc} 
\renewcommand{\arraystretch}{1.2} 
\scriptsize{
\begin{tabular}{cccccc}
\hline
&\multicolumn{5}{c}{$k=1$}\\
\hline
$\epsilon$&$m$&$ M$&$ a_{mP}$&$  a_{mP}/d_{mP}$& $E$\\
\hline
$   5.000\times 10^{-1} $& $1$& $  1.331\times 10^{3}$& $  3.722\times 10^{-4}$ &$  5.445\times 10^{-2}$&$  1.218\times 10^{-2}$ \\
$   5.000\times 10^{-4} $& $6$& $  2.875\times 10^{5}$& $  1.861\times 10^{-6}$ &$  9.147\times 10^{-3}$&$  3.673\times 10^{-4}$ \\
\hline \hline
&\multicolumn{5}{c}{$k=5$}\\
\hline
$\epsilon$&$m$&$ M$&$ a_{mP}$&$  a_{mP}/d_{mP}$& $E$\\
\hline
$   5.000\times 10^{-1} $& $1$& $ 1.331\times 10^{3}$& $  3.200\times 10^{-6}$ &$  2.196\times 10^{-3}$&$  1.05\times 10^{-4}$ \\
$   5.000\times 10^{-5} $& $8$& $  6.815\times 10^{5}$& $  6.650\times 10^{-9}$ &$  2.745\times 10^{-4}$&$  1.756\times 10^{-6}$ \\
\hline \hline
\end{tabular}}
\caption{ Example 4 }
\end{table}


\begin{thebibliography}{00}
\bibitem{RAMM1}
A.G. Ramm, Many body wave scattering by small bodies and
applications, J. Math. Phys., 48, N10, 103511, (2007).

\bibitem{RAMM2}
A.G. Ramm, Material with the desired refraction coefficients can be
made by embedding small particles, Phys. Lett. A, 370, N5-6, (2007),
522-527.

\bibitem{RAMM3}
A.G. Ramm, A recipe for making materials with negative refraction in
acoustic, Phys. Lett. A, 372/13, (2008), 2319-2321.

\bibitem{RAMM4}
A.G. Ramm, Wave scattering by many small particles embedded in a
medium, Phys. Lett. A, 372/17, (2008), 3064-3070.
\bibitem{RAMM563}
Ramm, A. G., A collocation method for solving integral equations,
Internat. Journ. of Comput. Sci. and Math., 2(3), (2009)
222-228.
\end{thebibliography}
\end{document}